\documentclass[12pt, a4paper]{article}
\usepackage{}

\usepackage{amsfonts}
\usepackage{bbm}
\usepackage{mathrsfs}
\usepackage{latexsym}
\usepackage{graphicx}
\usepackage{amssymb}
\usepackage{amsmath}
\usepackage{color,xcolor}
\usepackage{mathtools}

\newtheorem{theorem}{Theorem}[section]
\newtheorem{lemma}[theorem]{Lemma}
\newtheorem{corollary}[theorem]{Corollary}

\title{{\Large \bf  The Sombor index of trees and unicyclic graphs with given maximum degree
\thanks{Supported by the National Natural Science Foundation of China (No. 11771443).}~}}
\author{Ting Zhou, Zhen Lin\thanks{Corresponding author. E-mail addresses: tb19080009b1@cumt.edu.cn(T. Zhou), lnlinzhen@163.com (Z. Lin), miaolianying@cumt.edu.cn (L. Miao).}, Lianying Miao \\
{\footnotesize School of Mathematics, China University of Mining and Technology,}\\ {\footnotesize  Xuzhou, 221116, Jiangsu, P.R.
China}\\
}

\date{}
\begin{document}
\openup 1.0\jot
\date{}\maketitle
\begin{abstract}
Let $d_G(v)$ be the degree of the vertex $v$ in a graph $G$. The Sombor index of $G$ is defined as $SO(G) =\sum_{uv\in E(G)}\sqrt{d^2_G(u)+d^2_G(v)}$, which is a new degree-based topological index introduced by Gutman. Let $\mathscr{T}_{n,\Delta}$ and $\mathscr{U}_{n,\Delta}$ be the set of trees and unicyclic graphs with $n$ vertices and maximum degree $\Delta$, respectively. In this paper, the tree and the unicyclic graph with minimum Sombor index among $\mathscr{T}_{n,\Delta}$ and $\mathscr{U}_{n,\Delta}$ are characterized.

\bigskip


\end{abstract}
\baselineskip 20pt

\section{\large Introduction}
Let $G$ be a simple undirected graph with vertex set $V(G)$ and edge set $E(G)$.
For $v\in V(G)$, $N_G(v)$ denotes the set of all neighbors of $v$, and $d_G(v)=|N_G(v)|$ denotes the degree of vertex $v$ in $G$. Denote by $\Delta(G)$ (or $\Delta$) the maximum degrees of the vertices of $G$. A pendant vertex of $G$ is a vertex of degree one. Denote by $S_n$, $P_{n}$ and $C_n$ the star, path and cycle with $n$ vertices, respectively. Let $l(P_n)=|E(P_n)|$ be the length of $P_n$. Let $T_{n,\,\Delta}$, shown in Fig. 1.1, be the tree obtained by attaching a pendant edge to each of certain $n-\Delta-1$ non-central vertices of $S_{\Delta}$, and let $U_{n,\,\Delta}$, shown in Fig. 1.1, be the unicyclic graph obtained by attaching $2\Delta-n+1$ pendant vertices and $n-\Delta-1$ paths $P_2$ to one vertex of a cycle $C_3$.

A spider is a tree with at most one vertex of degree more than two, the unique vertex called the hub of the spider. A leg of a spider is a path from the hub to one of the leaves. Let $S(a_1,a_2,\ldots,a_k)$ be a spider with $k$ legs $P^1,P^2,\ldots,P^k$ for which the lengths $l(P^i)=a_i$ for $1\leq i\leq k$. Note that $T_{n,\,\Delta}$ is also a spider. For convenience, denote by $T_\Delta$ a spider with length of all $\Delta$ legs greater than 2, shown in Fig. 1.1. Let $U_\Delta$ be a unicyclic graph obtained by attaching $\Delta-2$ paths of length at least 2 to a cycle, shown in Fig. 1.1.

\vskip 5mm
\begin{figure}[!hbpt]
\begin{center}
\includegraphics[scale=0.43]{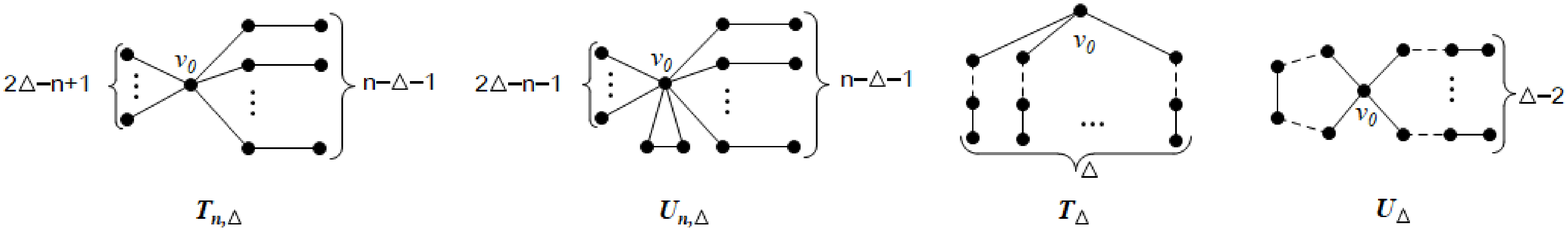}\\
Fig. 1.1~ Graphs $T_{n,\Delta}$, $U_{n,\Delta}$, $T_{\Delta}$, $U_{\Delta}$.
\end{center}\label{fig01}
\end{figure}

The Sombor index of a graph $G$ is defined as $SO(G) =\sum_{uv\in E(G)}\sqrt{d^2_G(u)+d^2_G(v)}$, which is a novel vertex-degree-based molecular structure descriptor proposed by Gutman \cite{G}. Red\v{z}epovi\'{c} \cite{R} showed that the Sombor index may be used successfully on modeling thermodynamic properties of compounds due to the fact that the Sombor index has satisfactory prediction potential in modeling entropy and enthalpy of vaporization of alkanes. Das et al. \cite{DCC} and Wang et al. \cite{WMLF} obtained the relations between the Sombor index and some other well-known degree-based descriptors, such as the first Zagreb index, the second Zagreb index, the forgotten topological index and so on. For other related results, one may refer to \cite{G1, GA, K, KG, MMM} and the references therein.

The extremal value problem of the topological index is of interest in mathematical chemistry. The investigation of extremal value of the Sombor index of graphs has quickly received much attention. Gutman \cite{G} obtained extremal values of the Sombor index among the set of (connected) graphs and the set of trees. Cruz et al. \cite{CGR} studied the Sombor index of chemical graphs, and characterized the graphs extremal with respect to the Sombor index over the following sets: chemical graphs, chemical trees, and hexagonal systems.
Deng et al. \cite{DTW} obtained a sharp upper bound for the Sombor index among all molecular trees with fixed numbers of vertices, and characterized those molecular trees achieving the extremal value. Liu \cite{L} determined the first fourteen minimum chemical trees, the first four minimum
chemical unicyclic graphs, the first three minimum chemical bicyclic graphs, the first seven minimum chemical tricyclic graphs. R\'{e}ti et al. \cite{RDA} characterized graphs with the maximum Sombor index in the classes of all connected unicyclic, bicyclic, tricyclic, tetracyclic, and pentacyclic graphs of a fixed order. Lin et al. \cite{LMZ} obtained lower and upper bounds on the spectral radius, energy and Estrada index of the Sombor matrix of graphs, and characterized the respective extremal graphs.

The purpose of this paper is to study the extremal value problem of Sombor index of trees and unicyclic graphs with given maximum degree. The following theorems are showed.

\begin{theorem}\label{th1,1} 
Let $n\geq 7$ and $T\in\mathscr{T}_{n,\Delta}$, where $3\leq \Delta\leq\ n-2$.

{\normalfont (i)} If $3\leq\Delta\leq\lfloor\frac{n-1}{2}\rfloor$, then
$$SO(T)\geq \Delta\sqrt{\Delta^2+4}+\sqrt{8}(n-2\Delta-1)+\sqrt{5}\Delta$$
with equality if and only if $T\cong T_\Delta$.

{\normalfont (ii)} If $\lfloor\frac{n-1}{2}\rfloor<\Delta\leq n-2$, then
$$SO(T)\geq (n-\Delta-1)\sqrt{\Delta^2+4}+(2\Delta-n+1)\sqrt{\Delta^2+1}+\sqrt{5}(n-\Delta-1)$$
with equality if and only if $T\cong T_{n,\,\Delta}$.
\end{theorem}

\begin{corollary}\label{cor1,1} 
Let $T$ be a chemical tree with $n\geq 7$ vertices. If $\Delta=3$ or $4$, then
$$SO(T)\geq 2\sqrt{2}n+3\sqrt{13}+3\sqrt{5}-14\sqrt{2}\quad \text{or}\quad SO(T)\geq 2\sqrt{2}n-6\sqrt{2}$$
with equality if and only if $T\cong T_3$ or $T\cong T_4$.
\end{corollary}

\begin{theorem}\label{th1,2} 
Let $n\geq 5$ and $U\in\mathscr{U}_{n,\Delta}$, where $3\leq\Delta\leq n-2$.

{\normalfont (i)} If $3\leq\Delta\leq\lfloor\frac{n+1}{2}\rfloor$, then
$$SO(U)\geq \Delta\sqrt{\Delta^2+4}+\sqrt{8}(n-2\Delta+4)+\sqrt{5}(\Delta-2)$$
with equality if and only if $U\cong U_\Delta$.

{\normalfont (ii)} If $\lfloor\frac{n+1}{2}\rfloor<\Delta\leq n-2$, then
$$SO(U)\geq (n-\Delta+1)\sqrt{\Delta^2+4}+(2\Delta-n-1)\sqrt{\Delta^2+1}+\sqrt{5}(n-\Delta-1)+\sqrt{8}$$
with equality if and only if $U\cong U_{n,\,\Delta}$.
\end{theorem}

\begin{corollary}\label{cor1,2} 
Let $T$ be a chemical unicyclic graph with $n\geq 5$ vertices. If $\Delta=3$ or $4$, then
$$SO(T)\geq 2\sqrt{2}n+3\sqrt{13}+\sqrt{5}-4\sqrt{2}\quad \text{or}\quad SO(T)\geq 2\sqrt{2}n+10\sqrt{5}-8\sqrt{2}$$
with equality if and only if $T\cong U_3$ or $T\cong U_4$.
\end{corollary}

\section{\large  Preliminaries}

The distance $d_G(u,v)$ between two vertices $u$, $v$ of $G$ is the length of one of the shortest $(u,v)$-path in $G$. Denote by $\overline{G}$ the complement of $G$ and by $P_{uv}$ the path between vertices $u$ and $v$. Let $G-u$ denote the graph that arises from a graph $G$ by deleting the vertex $u \in V(G)$ and all the edges incident with $u$. Let $G-uv$ denote the graph that arises from $G$ by deleting the edge $uv\in E(G)$. Similarly, $G+uv$ is the graph that arises from $G$ by adding an edge $uv\notin E(G)$, where $u, v\in V(G)$.

\begin{lemma}{\bf (\cite{KS})}\label{le2,1} 
Let $U\subseteq \mathbb{R}$ be an open interval and $f:U\rightarrow U$ a convex function. Let $a_1\geq a_2\geq\ldots\geq a_n$ and $b_1\geq b_2\geq\ldots\geq b_n$ be such elements in $U$ that inequalities $a_1+a_2+\ldots+a_n\geq b_1+b_2+\ldots+b_i$ hold for every $i\in\{1,2,\ldots,n\}$ and equality holds for $i=n$. Then, $f(a_1)+f(a_2)+\cdots+f(a_n)\geq f(b_1)+f(b_2)+\cdots+f(b_n)$.
\end{lemma}

\begin{lemma}\label{le2,2} 
For $w_0,\ x_0\in V(G)$ (where $w_0$, $x_0$ are not necessarily distinct), suppose that $w_1 w_2 \ldots w_k$, $x_1 x_2 \ldots x_h$ are two path components in $G-w_0$ and $G-x_0$, respectively, where $k,h\geq1$ and $w_k$, $x_h$ are pendant vertices of $G$. If $d_G(x_0)=s\geq3$, $N_G(x_0)=\{x_1,v_1,v_2,\ldots,v_{s-1}\}$, $d_G(v_i)\geq1$ for $1\leq i\leq s-1$, let $G'$ be a new graph with vertex set $V(G')=V(G)$ and edge set $E(G')=G-x_0x_1+w_kx_1$, see Fig. 1.2. When $w_0=x_0$, $G'$ is said to be obtained by running graph transformation $\mathbf{A}_1$ of $G$; when $w_0\neq x_0$, $G'$ is said to be obtained by running graph transformation $\mathbf{A}_2$ of $G$. Then $SO(G)>SO(G')$.

\begin{figure}[!hbpt]
\begin{center}
\includegraphics[scale=0.34]{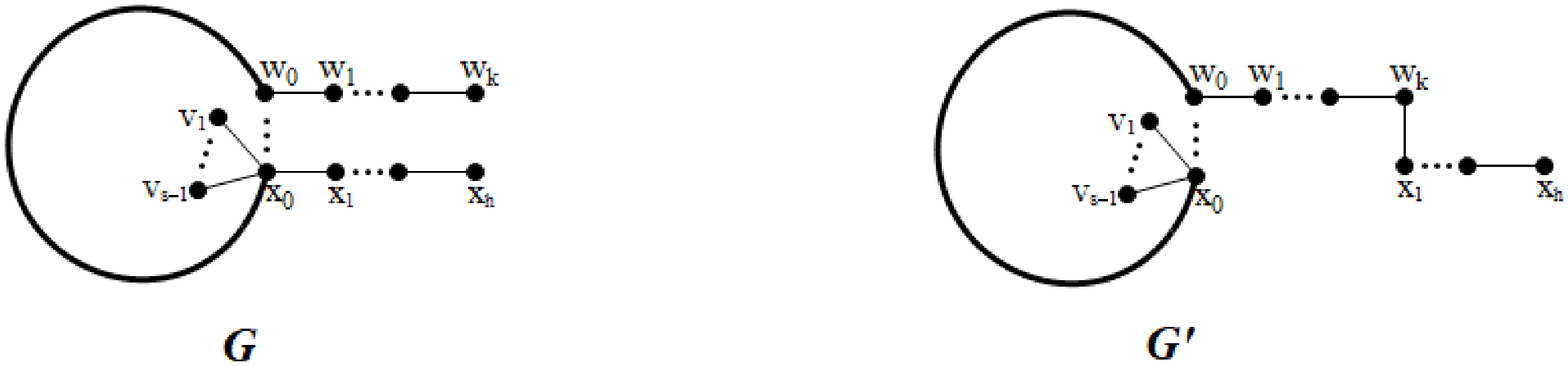}\\
Fig. 1.2 An illustration of Lemma \ref{le2,2}.
\end{center}\label{fig02}
\end{figure}
\end{lemma}

\begin{proof}From given conditions, $d_{G}(x_0)=s$, $d_{G'}(x_0)=s-1$, $d_{G}(w_k)=1$, $d_{G'}(w_k)=2$ and for any $v\in V(G)\setminus\{x_0,w_k\}$, $d_G(v)=d_{G'}(v)$. Since $d_G(x_0)=s\geq3$, by Lemma \ref{le2,1}, we have
\begin{align*}
 & SO(G')-SO(G) \\
= {}& \sum_{uv\in E(G)}\sqrt{d^2_{G}(u)+d^2_{G}(v)}-\sum_{uv\in E(G')}\sqrt{d^2_{G'}(u)+d^2_{G'}(v)}\\
= {} & \sum_{i=1}^{s-1}\Big[\sqrt{d^2_{G}(x_0)+d^2_{G}(v_i)}-\sqrt{d^2_{G'}(x_0)+d^2_{G}(v_i)}\Big]+\sqrt{d^2_{G}(x_0)+d^2_{G}(x_1)}\\
{} & +\sqrt{d^2_{G}(w_{k-1})+d^2_{G}(w_k)}-\sqrt{d_{G'}^2(w_{k-1})+d^2_{G'}(w_k)}-\sqrt{d_{G'}^2(w_{k})+d^2_{G'}(x_1)}.
\end{align*}

(i) If $w_0=x_0$, $k=h=1$, and $x_1=v_{s-1}$, then
\begin{align*}
 & SO(G')-SO(G) \\
= {}& \sum_{i=1}^{s-2}\Big[\sqrt{s^2+d^2_{G}(v_i)}-\sqrt{(s-1)^2+d^2_{G}(v_i)}\Big]+2\sqrt{s^2+1}-\sqrt{(s-1)^2+4}-\sqrt{5}\\
>{} & 2\sqrt{s^2+1}-\sqrt{(s-1)^2+4}-\sqrt{5}\\
>{} & 0.
\end{align*}

(ii) If $w_0\neq x_0$ and $k=h=1$, then
\begin{align*}
 & SO(G')-SO(G) \\
= {}& \sum_{i=1}^{s-1}\Big[\sqrt{s^2+d^2_{G}(v_i)}-\sqrt{(s-1)^2+d^2_{G}(v_i)}\Big]+\sqrt{s^2+1}\\
& {}+\sqrt{d^2_{G}(w_0)+1}-\sqrt{d_G^2(w_0)+4}-\sqrt{5}\\
>{} & \sqrt{s^2+1}+\sqrt{d^2_{G}(w_0)+1}-\sqrt{d_G^2(w_0)+4}-\sqrt{5}\\
>{} & 0.
\end{align*}

(iii) If $k\geq2$, $h=1$, then
\begin{align*}
 & SO(G')-SO(G) \\
= {}& \sum_{i=1}^{s-1}\Big[\sqrt{s^2+d^2_{G}(v_i)}-\sqrt{(s-1)^2+d^2_{G}(v_i)}\Big]+\sqrt{s^2+1}+\sqrt{5}-\sqrt{8}-\sqrt{5}\\
>{} & \sqrt{s^2+1}-\sqrt{8}\\
>{} & 0.
\end{align*}

(iv) If $k,h\geq2$, then
\begin{align*}
 & SO(G')-SO(G) \\
= {}& \sum_{i=1}^{s-1}\Big[\sqrt{s^2+d^2_{G}(v_i)}-\sqrt{(s-1)^2+d^2_{G}(v_i)}\Big]+\sqrt{s^2+4}+\sqrt{5}-2\sqrt{8}\\
>{} & \sqrt{s^2+4}+\sqrt{5}-2\sqrt{8}\\
>{} & \sqrt{13}+\sqrt{5}-2\sqrt{8}\\
>{} & 0.
\end{align*}

Combining the above arguments, we have the proof.  $\Box$

\end{proof}

\section{\large The proof of Theorem \ref{th1,1}}

In this section, we determine the trees with the minimum Sombor index among $n$-vertex trees with maximum degree $\Delta$. When $\Delta=2$, $\mathscr{T}_{n,\Delta}=\{P_n\}$ and when $\Delta=n-1$, $\mathscr{T}_{n,\Delta}=\{S_n\}$. From now on, we assume $3\leq\Delta\leq n-2$.

\noindent{\bf Proof of Theorem \ref{th1,1}} Let $T$ be an $n$-vertex tree with maximum degree $\Delta$ that minimize the Sombor index and let $v_0$ be a $\Delta$-vertex of $T$. We will show the following Claims 1-3, which, put together, will get our proof.

\noindent{{\bf Claim 1.} $T$ is a spider.}

\begin{proof}
Suppose $T$ is not a spider. There exists $v\in V(T)\setminus\{v_0\}$, such that $d_T(v)\geq3$. Then we can get a new tree $T_1\in\mathscr{T}_{n,\Delta}$ by running graph transformation $\mathbf{A}_1$ on $v$. By Lemma \ref{le2,2}, $SO(T_1)<SO(T)$, which contradicts the choice of $T$. Thus $T$ is a spider. $\Box$
\end{proof}

Let $T=S(a_1,a_2,\ldots,a_\Delta)$ with $\Delta$ legs $P^1,P^2,\ldots,P^\Delta$, and the lengths $l(P^i)=a_i$ for $1\leq i\leq\Delta$. Without loss of generality, we assume $a_1 \geq a_2\geq\ldots\geq a_\Delta$.

\noindent{{\bf Claim 2.} If $3\leq\Delta\leq\lfloor\frac{n-1}{2}\rfloor$, then $a_\Delta\geq2$.}

\begin{proof} Suppose $a_\Delta=1$. Since $n-1\geq2\Delta$, we have $a_1\geq3$. Let $P^1=v_0v^1_1v^1_2\ldots v^1_{s}$ where $s\geq3$ and $P_\Delta=v_0v_\Delta$.
Let $T_2=T-v^1_{s-1}v^1_{s}+v_{\Delta}v^1_{s}$, by Lemma \ref{le2,2}, we have
$SO(T)-SO(T_2)=\sqrt{\Delta^2+1}+\sqrt{8}-\sqrt{\Delta^2+4}-\sqrt{5}>0,$
which contradicts the choice of $T$. Thus $a_\Delta\geq2$, that is, $T\cong T_\Delta$. $\Box$
\end{proof}

\noindent{{\bf Claim 3.} If $\lfloor\frac{n-1}{2}\rfloor<\Delta\leq n-2$, then $a_1\leq2$.}

\begin{proof}
Suppose $a_1\geq3$. Since $n-1<2\Delta$, we have $a_\Delta=1$. Let $P^1=v_0v^1_1v^1_2\ldots v^1_{s}$ where $s\geq3$ and $P_\Delta=v_0v_\Delta$.
Similar to the proof of Claim 2, let $T_3=T-v^1_{s-1}v^1_{s}+v_{\Delta}v^1_{s}$, then $SO(T)>SO(T_3)$, which contradicts the choice of $T$. Thus $a_1\leq2$, that is, $T\cong T_{n,\,\Delta}$. $\Box$
\end{proof}

By direct calculations, we get $SO(T_\Delta)=\Delta\sqrt{\Delta^2+4}+\sqrt{8}(n-2\Delta-1)+\sqrt{5}\Delta$ and $SO(T_{n,\,\Delta})=(n-\Delta-1)\sqrt{\Delta^2+4}+(2\Delta-n+1)\sqrt{\Delta^2+1}+\sqrt{5}(n-\Delta-1)$.
This completes the proof of Theorem \ref{th1,1}.  $\Box$

\section{\large The proof of Theorem \ref{th1,2}}

In this section, we determine the unicyclic graphs with the minimum Sombor index among $n$-vertex unicyclic graphs with maximum degree $\Delta$. When $\Delta=2$, $\mathscr{U}_{n,\Delta}=\{C_n\}$ and when $\Delta=n-1$, $\mathscr{U}_{n,\Delta}=\{S_n+e\}$ where $e$ is a edge in $\overline{S_n}$. Next, we assume $3\leq\Delta\leq n-2$.

\noindent{\bf Proof of Theorem \ref{th1,2}} Let $U$ be an $n$-vertex unicyclic graph with maximum degree $\Delta$ that minimize the Sombor index. Suppose $C$ is the unique cycle of $U$. If there exist a $\Delta$-vertex on $C$, then we chose it and denote by $v_0$; otherwise, chose any $\Delta$-vertex, also denote by $v_0$. First, we assume $v_0\notin V(C)$. Then there is a vertex $v\in V(C)$ such that $d_U(v,v_0)=\min\{d_U(u,v_0)|u\in V(C)\}$, clearly, $d_U(v)\geq3$. We will show the following Claims 1-5, which, put together, will get our proof.

\noindent{{\bf Claim 1.} For any $u\in V(U)\setminus\{v_0,v\}$, $d_U(u)\leq2$.}

\begin{proof}
If the claim is not true, there are three cases:

{\bf Case i.} there exists $u\in V(U)\setminus (V(C)\cup P_{vv_0})$ such that $d_U(u)\geq3$. Then we can get a new unicyclic graph $U_1\in\mathscr{U}_{n,\Delta}$ by running graph transformation $\mathbf{A}_1$ on $u$. By Lemma \ref{le2,2}, $SO(U)>SO(U_1)$, which contradicts the choice of $U$.

{\bf Case ii.}  there exists $u\in (V(C)\cup P_{vv_0})\setminus\{v,v_0\}$ such that $d_U(u)\geq4$. There are at least two paths starting from $u$ to pendant vertices of $U$, then by running transformation $\mathbf{A}_1$ on $u$, we can get a contradiction.

{\bf Case iii.} there exists $u\in V(C)\cup P_{vv_0}\setminus\{v,v_0\}$ such that $d_U(u)=3$. Let $uu_1u_2\ldots u_s$ be the path from $u$ to pendant vertex $u_s$, where $s\geq1$ and $d_U(u_i)=2$ for $1\leq i\leq s-1$. And let $v_0v_1v_2\ldots v_t$ be one of the paths from $v_0$ to pendant vertex $v_t$, where $t\geq1$ and $d_U(v_i)=2$ for $1\leq i\leq t-1$. Let $U_2=U-uu_1+v_su_1$, then $U_2$ is obtained by running transformation $\mathbf{A}_2$ from $U$ and $U_2\in\mathscr{U}_{n,\Delta}$. By Lemma \ref{le2,2}, we have $SO(U)>SO(U_2)$, which contradicts the choice of $U$.

Combining the above cases, we have $d_U(u)\leq2$ for any $u\in V(U)\setminus\{v_0,v\}$. $\Box$
\end{proof}

\noindent{{\bf Claim 2.} $d_U(v)\leq3$.}

\begin{proof}
If $d_U(v)\geq 5$, there are at least two paths starting from $v$ to pendant vertices of $U$, similarly, by running transformation $\mathbf{A}_1$, we can get a contradiction. If $d_U(v)=4$, by running transformation $\mathbf{A}_2$, we can also get a contradiction. $\Box$
\end{proof}

Denote by $v_1,v_2,\ldots,v_\Delta$ the neighbors of $v_0$, where $v_1\in V(P_{vv_0})$.

\noindent{{\bf Claim 3.} $v=v_0$, that is to say, there must be $v_0\in V(C)$.}

\begin{proof}
For otherwise, we can get a new unicyclic graph $U_3=U-\{v_0v_i|2\leq i\leq\Delta-2\}+\{vv_i|2\leq i\leq\Delta-2\}$ and $U_3\in\mathscr{U}_{n,\Delta}$. For $\Delta-1\leq i\leq\Delta$, $d_U(v_i)\leq2$, we have
\begin{align*}
 & SO(U)-SO(U_3) \\
= {}& 2\sqrt{3^2+2^2}+\sum_{i=\Delta-1}^{\Delta}\sqrt{\Delta^2+d^2_U(v_i)}-2\sqrt{\Delta^2+2^2}-\sum_{i=\Delta-1}^{\Delta}\sqrt{3^2+d^2_U(v_i)}\\
={} & \Big[2\sqrt{3^2+2^2}-\sum_{i=\Delta-1}^{\Delta}\sqrt{3^2+d^2_U(v_i)}\Big]-\Big[2\sqrt{\Delta^2+2^2}-\sum_{i=\Delta-1}^{\Delta}\sqrt{\Delta^2+d^2_U(v_i)}\Big]\\
\geq {} & 0.
\end{align*}

We have now in $U_3$, $d_{U_3}(v)=\Delta$ and $d_{U_3}(v_0)=3$, then there are at least two paths starting from $v_0$ to pendant vertices of $U_3$, similarly, by running transformation $\mathbf{A}_1$ on $v_0$, we can get a contradiction, which contradicts the choice of $U$. Thus $v_0\in V(C)$. $\Box$
\end{proof}

By Claims 1-3, $U$ is a unicyclic graph obtained by attaching and $\Delta-2$ paths to a cycle $C$. Let $v_1,v_2\in V(C)$. Similar to the proof of Theorem \ref{th1,1}, denote by $P^i$ the path from $v_0$ to a pendant vertex of $U$ and $v_i\in P^i$, where $3\leq i\leq\Delta$. Without loss of generality, we can assume that $l(P^3)\geq l(P^4)\geq\ldots\geq l(P^\Delta)$.

\noindent{{\bf Claim 4.} If $3\leq\Delta\leq\lfloor\frac{n+1}{2}\rfloor$, then $l(P^\Delta)\geq 2$.}

\begin{proof}
Suppose that $l(P^\Delta)=1$. Since $n-3\geq2(\Delta-2)$, we have the following two cases:

{\bf Case i.}  $l(P^3)\geq3$. Let $P^\Delta=v_0v_\Delta$ and $P^3=v_0v_3v^3_2\ldots v^3_{s}$, where $s\geq3$. Let $U_3=U-v^3_{s-1}v^3_{s}+v^{\Delta}v^3_{s}$, by Lemma \ref{le2,2}, we have
$SO(U)-SO(U_3)=\sqrt{\Delta^2+1}+\sqrt{8}-\sqrt{\Delta^2+4}-\sqrt{5}>0,$
which contradicts the choice of $U$.

{\bf Case ii.} $l(P^3)=2$ and $|E(C)|\geq4$. Let $x,y,z$ be three vertices on $C$ different from $v_0$ such that $xy,yz\in E(C)$.
Let $U_4=U-xy-yz+xz+v_{\Delta}y$, we have
$SO(U)-SO(U_4)=\sqrt{\Delta^2+1}+\sqrt{8}-\sqrt{\Delta^2+4}-\sqrt{5}>0,$
which contradicts the choice of $U$.

Thus $l(P^\Delta)\geq 2$, that is, $U\cong U_\Delta$. $\Box$
\end{proof}

\noindent{{\bf Claim 5.} If $\lfloor\frac{n+1}{2}\rfloor<\Delta\leq n-2$, then $l(P^3)\leq2$ and $|E(C)|=3$.}

\begin{proof}
Note that $n-3<2(\Delta-2)$. If $l(P^3)\geq3$ or $|E(C)|\geq4$, then $l(P^\Delta)=1$. Similar to the proof of Claim 4, we can get a contradiction. Thus $l(P^3)\leq2$ and $|E(C)|=3$, that is, $U\cong U_{n,\Delta}$. $\Box$
\end{proof}

By direct calculations, we get $SO(U_\Delta)=\Delta\sqrt{\Delta^2+4}+\sqrt{8}(n-2\Delta+4)+\sqrt{5}(\Delta-2)$ and $SO(U_{n,\Delta})=(n-\Delta+1)\sqrt{\Delta^2+4}+(2\Delta-n-1)\sqrt{\Delta^2+1}+\sqrt{5}(n-\Delta-1)+\sqrt{8}$.  This completes the proof of Theorem \ref{th1,2}. $\Box$

\vskip 5mm

\small {

}


\begin{thebibliography}{99}


\bibitem{CGR} R. Cruz, I. Gutman, J. Rada, Sombor index of chemical graphs, Appl. Math. Comput. 399 (2021) 126018.%

\bibitem{DCC} K.Ch. Das, A.S. \c{C}evik, I.N. Cangul, Y. Shang, On Sombor Index, Symmetry 13 (2021) 140.%

\bibitem{DTW} H. Deng, Z. Tang, R. Wu, Molecular trees with extremal values of Sombor indices, Int J Quantum Chem. DOI: 10.1002/qua.26622.

\bibitem{G} I. Gutman, Geometric approach to degree-based topological indices: Sombor indices, MATCH Commun. Math. Comput. Chem. 86 (2021) 11-16.%

\bibitem{G1} I. Gutman, Some basic properties of Sombor indices, Open J. Discret. Appl. Math.  4 (2021) 1-3.%

\bibitem{GA} N. Ghanbari, S. Alikhan, Sombor index of certain graphs, arXiv:2102.10409v1.

\bibitem{K} V.R. Kulli, Sombor indices of certain graph operators, International Journal of Engineering Sciences \& Research Technology, 10 (2021) 127-134.%

\bibitem{KG} V.R. Kulli, I. Gutman, Computation of Sombor Indices of Certain Networks, International Journal of Applied Chemistry, 8 (2021) 1-5.%

\bibitem{KS} J. Karamata, Sur une in\'{e}galit\'{e} relative aux fonctions convexes, Publ. Math. Univ. Belgrade 1 (1932) 145-148.%

\bibitem{L} H. Liu, Ordering chemical graphs by their Sombor indices, arXiv:2103.05995v1.

\bibitem{LMZ} Z. Lin, L. Miao, T. Zhou, On the spectral radius, energy and Estrada index of the Sombor matrix of graphs, arXiv:submit/3627920.

\bibitem{MMM} I. Milovanovi\'{c}, E. Milovanovi\'{c}, M. Mateji\'{c}, On some mathematical properties of sombor indices, Bull. Int. Math. Virtual Inst. 11 (2021) 341-353.

\bibitem{R} I. Red\v{z}epovi\'{c}, Chemical applicability of Sombor indices, J. Serb. Chem. Soc. https://doi.org/10.2298/JSC201215006R.

\bibitem{RDA} T. R\'{e}ti, T. Do\v{s}li\'{c}, A. Ali, On the Sombor index of graphs, Contrib. Math. 3 (2021) 11-18.%

\bibitem{WMLF} Z. Wang, Y. Mao, Y. Li, B. Furtula, On relations between Sombor and other degree-based indices, J. Appl. Math. Comput. https://doi.org/10.1007/s12190-021-01516-x.





\end{thebibliography}
\end{document}